\documentclass[11pt,a4paper, reqno]{amsart}  

\usepackage{latexsym}

\usepackage{amsmath,amsfonts,amssymb,amsthm,amscd,epsfig}  
\usepackage[english]{babel}  
\usepackage{indentfirst}  
\usepackage[mathscr]{eucal}  
\usepackage{graphics}  
  
\numberwithin{equation}{section}  
  
\renewcommand{\a}{\alpha}  
\renewcommand{\b}{\beta}  
\newcommand{\g}{\gamma}

\renewcommand{\l}{\lambda}

\newcommand{\e}{\varepsilon}

\newcommand{\R}{{\mathbb R}}

\newcommand{\1}{{1\hspace{-1.4 mm}1}}
\newcommand{\car}{{1\hspace{-1.4 mm}1}}

\newcommand{\Sc}{{\mathbb S}}

\newcommand{\Hc}{{\mathcal H}}  
  
\newcommand{\Lc}{{\mathcal L}}

\def\oc{{\overline c}}
\def\uc{{\underline c}}
\def\etab{{\overline \eta}}

\newtheorem{theorem}{Theorem}[section]  
  
\newtheorem{lemma}[theorem]{Lemma}

\newtheorem{definition}[theorem]{Definition}  
\newtheorem{proposition}[theorem]{Proposition}  

\theoremstyle{definition}

\theoremstyle{remark}  
\newtheorem{remark}[theorem]{Remark}

\title{Uniqueness Results for Nonlocal Hamilton-Jacobi Equations}

\begin{document}

\author{Guy Barles   
\and Pierre Cardaliaguet 
 \and {Olivier Ley} \and Aur\'elien Monteillet}

\address{ (G. Barles, O. Ley)  
Laboratoire de Math\'ematiques et Physique Th\'eo\-ri\-que \\
F\'ed\'eration Denis Poisson \\
Universit\'e de Tours \\
Parc de Grandmont, 37200 Tours, France \\ {\tt \{barles,ley\}@lmpt.univ-tours.fr}}
\address{
(P. Cardaliaguet, A. Monteillet) Laboratoire de Math\'ematiques \\ CNRS UMR 6205\\  Universit\'e de Brest \\ 6 Av. Le Gorgeu
BP 809, 29285 Brest, France \\
 {\tt \{pierre.car\-da\-liaguet, aurelien.monteillet\}@univ-brest.fr}} 

\begin{abstract} 
We are interested in nonlocal Eikonal Equations describing the evolution  of interfaces moving with a nonlocal, non monotone velocity. 
For these equations, only the existence of global-in-time weak solutions is available in some particular cases. 
In this paper, we propose a new approach for proving uniqueness of the solution 
when the front is expanding. This approach simplifies and extends existing results for dislocation dynamics. It 
also provides the first uniqueness result for a Fitzhugh-Nagumo system. The
 key ingredients are some new perimeter estimates for the evolving fronts 
 as well as some uniform interior cone property for these fronts.
\end{abstract} 

\thanks{This work was partially supported by the ANR (Agence Nationale de la Recherche) through MICA project (ANR-06-BLAN-0082)}
 
 

\maketitle

\section{Introduction}

In this article, we are interested in uniqueness results for different
types of problems which can be written as nonlocal Hamilton-Jacobi 
Equations of the following form:
\begin{equation}\label{dyseqn}
u_t=c[\car_{\{u\geq 0\}}](x,t)|D u|\quad \hbox {in  }\R^N \times (0,T)\; ,
\end{equation}
\begin{equation}\label{iddyseqn}
u(x,0)=u_0 (x) \quad \hbox {in  }\R^N \; ,
\end{equation}
where $T > 0$, the solution $u$ is a real-valued function, $u_t$ and$Du$ 
stand respectively for its time and space derivatives and $\car_A$ is the 
indicator function of a set $A$. Finally $u_0$ is a bounded, 
Lipschitz continuous function.

For any indicator function $\chi$ or more generally for any 
$\chi \in L^\infty$ with $0 \leq \chi \leq 1$ a.e., the function 
$c[\chi]$ depends on $\chi$ in a nonlocal way and, in the main 
examples we have in mind, it is obtained from $\chi$ through a 
convolution type procedure (either only in space or in space and 
time). In particular, in our framework, despite the fact that $\chi$ has no 
regularity neither in $x$ nor in $t$, $c[\chi]$ will be always 
Lipschitz continuous in $x$; on the contrary we impose no 
regularity with respect to $t$. More precisely we always 
assume in the sequel that, for any
$\chi$, the velocity $c=c[\chi]$ satisfies

\vspace*{2mm}\noindent{\bf (H1)} 
For all $x \in \R^N$, $t \mapsto c(x,t)$ is measurable and there 
exist ${C}, \uc, \oc >0$ such that, for all $x,y\in \R^N$
and $t\in [0,T],$
\begin{eqnarray}
& & |c(x,t)-c(y,t)|\leq {C}|x-y|, \nonumber \\
& & 0 < \uc \leq c(x,t)\leq \oc. \label{cborne}
\end{eqnarray}
We will come back to this assumption later on.

To give a first flavor of our main uniqueness results, we can 
point out the following key facts: Equation (\ref{dyseqn}) 
can be seen as the ``level-set approach''-equation 
associated to the motion of the front $\Gamma_t:=\{x:u(x,t)=0\}$ 
with the nonlocal velocity $c[\car_{\{u(\cdot ,t)\geq 0\}}].$ However, 
in the non-standard examples we consider, it is not only 
a nonlocal but also non-monotone ``geometrical'' equation; 
by non-monotone we mean that the inclusion principle, which 
plays a central role in the ``level-set approach'', does not 
hold and, therefore, the uniqueness of solutions cannot be proved 
via standard viscosity solutions methods.

In fact, the few uniqueness results which exist in the literature 
(see below) rely on $L^1$ type estimates on the solution; this 
is natural since one has to connect the continuous function $u$ 
and the indicator function $\car_{\{u\geq 0\}}$. 
The main estimates concern measures of sets of the type 
$\{x: a\leq u(x,t)\leq b\}$ for $a,b$ close to $0$. 
Whether or not the aforementioned estimate has to be uniform on time, or of integral 
type, strongly depends on the properties of the convolution kernel. 
In order to emphasize this fact, we are going to concentrate on 
two model cases: the first one is a dislocation type equation
(see Section \ref{sec:mod1})
in which the kernel belongs to $L^\infty$ while the second one is 
related to the Fitzhugh-Nagumo system arising in neural wave 
propagation or in chemical kinetics in which the kernel is 
essentially the kernel of the Heat Equation (see Section
\ref{sec:mod2}).
In that case, it  is not in $L^\infty$. 
The fact that the convolution kernel 
is, or is not, bounded is indeed {\it the} key difference here.

Before going further, let us give some
references: for the first model case (dislocation type 
equations), we refer the reader to 
Barles, Cardaliaguet, Ley and Monneau \cite{bclm07} where general 
results are provided for these equations. We point out---and we 
will come back to this fact later---that uniqueness in the 
non-monotone case was first obtained by 
Alvarez, Cardaliaguet and Monneau \cite{acm04} and then by 
Barles and Ley \cite{bl06} using different arguments; we also 
refer to Rodney, Le Bouar and Finel \cite{rlf03} for 
the physical background on these equations. 
The Fitzhugh-Nagumo system has been investigated in particular by  
Giga, Goto and Ishii \cite{ggi92}, and by Soravia and Souganidis \cite{ss96}. 
They provided a notion of weak solution for this system (see \eqref{system} below) and 
proved existence of such weak solutions. They also study the 
connections with the phase field model (a reaction-diffusion 
system which leads to such a front propagation model).
However the uniqueness question has been left open until now. \\

Let us return to the key steps to prove uniqueness 
for (\ref{dyseqn})-(\ref{iddyseqn}). A major issue is 
the properties of the solutions of the Eikonal equations of the form
\begin{equation}\label{chjb}
u_t = c(x,t)|D u| \quad\hbox{ in} \ \R^N\times (0,T),
\end{equation}
where $c$ is a continuous function, satisfying suitable assumptions. 
Of course, such partial differential equations 
appear naturally when considering 
$\car_{\{u\geq 0\}}$ as an a priori given function. We recall that this 
equation is related via the level-set approach to the motion 
of fronts with a $(x,t)$-dependent normal velocity $c(x,t)$ 
and to deal with compact fronts and to simplify matter, 
we assume that the initial datum satisfies the following conditions: the subset $\{u_0>0\}$
is non empty and there exists 
$R_0>0$ such that
\begin{equation} \label{cond-infinity}
u_0=-1 \quad \text{in} \; \R^N \setminus \bar{B}(0,R_0).
\end{equation}
This implies, in particular, that the initial front 
$\Gamma_0 = \{u_0 = 0\}$ is a non empty compact subset of $B(0,R_0)$.

Assumption {\bf (H1)} ensures existence and uniqueness 
of a solutions to \eqref{chjb} but we also assume that the 
function $c=c[\chi]$ is positive (and even strictly positive), 
together with

\vspace*{2mm}
\noindent{\bf (H2)} There exists $\eta_0 >0$ such that
\begin{eqnarray*}
-|u_0(x)|-|D u_0(x)|+\eta_0 \leq 0 \ {\rm in} \ \R^N \ 
{\rm in \ the \ viscosity \ sense.}
\end{eqnarray*}
The above assumption implies that the set $\{u=0\}$ has a zero Lebesgue 
measure (cf. Ley \cite{ley01}) which is an important 
property for our arguments. Indeed \cite{bclm07} provides 
a counter-example (even in a (quasi) monotone case) where 
fattening phenomena leads to a non-uniqueness feature for 
a nonlocal equation. In addition to this non-fattening 
property, a key consequence of {\bf (H1)}-{\bf (H2)} is 
a lower bound on the gradient $Du$ on a set 
$\{x: \, |u(x,t)|\leq \eta\}$ for a small enough $\eta$ 
(cf. \cite{ley01}).

We now concentrate on the estimates of the measures of the 
volume of sets like $\{a \leq u(\cdot,t) \leq b\}$ 
where $-\eta \leq a < b \leq \eta$. 
We first note that such estimates are related with 
perimeter estimates
of the  $\alpha$ level-sets of $u$  for $\alpha$ close to $0$ 
(typically $|\alpha|<\eta$): indeed, combining the co-area formula 
with the lower bound on the gradient of the solution, we obtain
\begin{eqnarray}
\int_{\R^N}\, \car_{\{a \leq u(\cdot,t) \leq b\}} dx & = & \int_a^b
\int_{\{u(\cdot,t)=s\}}|Du|^{-1} d{\mathcal H}^{n-1} ds \nonumber \\
& \leq &  \frac{b-a}{\etab}\sup_{a \leq s \leq b}\,{\rm Per}(\{u(\cdot,t)=s\})\; ,
\label{perim-est}
\end{eqnarray}
where $\etab$ is the lower bound on $|Du|$ on the set
$\{x: \, |u(x,t)|\leq \eta\}$.

In  \cite{acm04} and \cite{bl06}, perimeter estimates for the 
$\alpha$ level-sets of $u$ were obtained by using bounds on the 
curvatures of these sets. 
Although this approach was powerful, it has the drawback to require 
strong assumptions on the dependence in $x$ of 
$c[\chi]$ (typically a ${\mathcal C}^{1,1}$ regularity).
Unfortunately such strong regularity does not always hold: for 
instance it is not the case for the  Fitzhugh-Nagumo system.

The key contribution of this paper is to provide $L^1([0,T])$ 
or  $L^\infty([0,T])$ estimates  of the volume 
of the set $\{a \leq u(\cdot,t) \leq b\}$
 (or, almost equivalently, of the perimeter of the $\alpha$ 
level-sets of $u$) in situations where the velocity $c[\chi]$ 
is less regular in $x$. 
As a consequence we are able to prove new uniqueness results. 
 
For the dislocation dynamics model, our approach allows to 
relax the assumptions of \cite{acm04} and \cite{bl06} on the data. 
The proofs are also simpler, 
requiring only  $L^1([0,T])$ estimates and a mild regularity 
($c[\chi]$ is merely measurable in time and Lipschitz continuous in space). 
So the main conclusion here is that ``soft" estimates are sufficient 
provided the convolution kernel is in $L^\infty$.

On the contrary, for the Fitzhugh-Nagumo system, where the
convolution kernel is unbounded, these $L^1$-estimates are no more sufficient 
and the uniqueness proof rather requires heavy 
$L^\infty$-estimates on the perimeter. These estimates are obtained 
by establishing, through optimal control type arguments, that 
the set $\{x: u(x,t) >0\}$ satisfies a uniform ``interior cone property", 
from which we deduce (explicit) estimates on the perimeter. 

The paper is organized as follows: in Section~\ref{sec:existence}, 
we recall the notion of weak solution for \eqref{dyseqn} introduced 
in \cite{bclm07}.  In Section~\ref{sec:mod1} we prove uniqueness of 
the solution for the dislocation type equation, while we deal with 
the Fitzhugh-Nagumo case in Section~\ref{sec:mod2}. The main technical 
results of this paper are gathered in Section~\ref{auxil}: we recall 
here some useful results for the Eikonal Equation \eqref{chjb}, 
we show the interior cone property and deduce the uniform perimeter estimates.\\

\noindent\textbf{Aknowledgment.}
This work was supported by the contract ANR MICA ``Mouvements d'Interfaces, Calcul et Applications''.\\

\noindent\textbf{Notation.} In the sequel, $|\cdot|$ denotes the standard euclidean norm in 
$\R^N$, $B(x,R)$ (resp. $\bar{B}(x,R)$) is the open (resp. closed) 
ball of radius $R$ centered at $x \in \R^N$.  
We denote the essential supremum of $f\in L^\infty (\R^N)$ 
or $f\in L^\infty(\R^n\times (0,T))$ by $|f|_\infty.$ Finally, 
${\mathcal L}^n$ and ${\mathcal H}^n$ denote, respectively, the $n$-dimensional
Lebesgue and Hausdorff measures.

\section{Definition of weak solutions to \eqref{dyseqn}}  
\label{sec:existence}  

We will use the following definition of weak solutions introduced in \cite{bclm07}.

\begin{definition}  
Let $u: \R^N \times [0,T]:\to \R$ be a continuous function. We say that $u$ 
is a weak solution of \eqref{dyseqn}-\eqref{iddyseqn} if there exists $\chi \in 
L^{\infty}(\R^N \times [0,T];[0,1])$ such that
\begin{enumerate}  
\item $u$ is a $L^1$-viscosity solution of  
\begin{eqnarray} \label{eq-gene2}
\left\{
\begin{array}{ll}
 u_t(x,t)= c[\chi](x,t)|Du(x,t)| & {in}\
 \R^N\times (0,T), \\
 u(\cdot , 0)= u_0 & {in}\ \R^N.
\end{array}
\right.
\end{eqnarray}
\item For almost all $t \in [0,T]$, 
$$
\1_{\{ u(\cdot,t) > 0 \}} 
\leq \chi(\cdot,t) \leq \1_{\{ u(\cdot,t) \geq 0 \}}
\ \ \ { in} \ \R^N.
$$  
\end{enumerate}  
Moreover, we say that $u$ is a classical solution of \eqref{dyseqn}
if in addition, for almost all $t \in [0,T]$ and almost everywhere in $\R^N$,  
$$  
\1_{\{u(\cdot,t)>0\}}=\1_{\{u(\cdot,t)\geq0\}} \ .  
$$  
\end{definition}  

We refer  to \cite[Appendix]{bclm07} for basic definition and
properties of $L^1$-viscosity solutions and to
\cite{ishii85, nunziante90, nunziante92, bourgoing04a,bourgoing04b}
for a complete presentation of the theory. 


\section{Model problem 1: dislocation type equations}\label{sec:mod1}
  
In this section, we consider equations arising in dislocations
theory (cf. \cite{rlf03}) where, for all 
$\chi \in L^\infty (\R^N)$ or $L^1 (\R^N)$, $c[\chi]$ is defined by
\begin{equation}\label{cdisloc}
c[\chi](x,t)=(c_0 * \chi)(x,t)+c_1 (x,t)\ \hbox {in  }\R^N \times (0,T),
\end{equation}
where $c_0, c_1$ are given functions, satisfying suitable assumptions 
which are described later on and ``$*$'' 
stands for the usual convolution in $\R^N$ with respect to the space 
variable $x.$ Our main result below applies to slightly more general 
cases but the main interesting points appear on this model case.

We refer to \cite{bclm07} for a complete description of the 
characteristics and difficulties connected to \eqref{dyseqn} in this
case; as recalled in the introduction, not this  equation is not only 
nonlocal but it is also, in general, non-monotone, which 
means that the maximum principle (or, here, inclusion principle) 
does not hold and one cannot apply directly viscosity solutions' 
theory. Roughly speaking, a (more or less) direct use of viscosity 
solutions' theory requires that $c_0 \geq 0$ in $\R^N \times (0,T)$, 
an assumption which is not natural
in the dislocations' framework.

We use the following assumptions on $c_0$ and $c_1$.\\

\noindent{\bf (H3)} $c_0, c_1 \in {\mathcal C}^0(\R^N\times [0,T])$ 
and there exists a constant ${C}$ 
such that, for any $x,y\in \R^N$ and $t \in [0,T]$,
$$ |c_0(x,t) -c_0(y,t)| + |c_1(x,t) -c_1(y,t)| \leq {C}|x-y|\; .$$
Moreover, $c_0 \in {\mathcal C}^0([0,T];L^1 (\R^N))$ and  there exist $\uc,\oc >0$ 
such that, for any $x\in \R^N$ and $t \in [0,T]$,
\begin{eqnarray*}
& |c_0(x,t)| \leq \oc\; ,  &\\
 &0< \uc \leq -|c_0(\cdot ,t)|_{L^1} + c_1(x,t) \leq |c_0(\cdot ,t)|_{L^1} 
+ c_1(x,t) \leq \oc\; .&
\end{eqnarray*}

This assumption ensures that the velocity $c[\chi]$ in \eqref{cdisloc}
satisfies {\bf (H1)} with constants independent of $0\leq \chi\leq 1$
with compact support in some fixed ball (see Step 1 in the proof of
Theorem~\ref{unicite-nl}).
Assumption {\bf (H3)} can be slightly relaxed (and in particular
localized) using that the front remains
in a bounded region of $\R^N$.
Note that, in contrast to \cite{bclm07}, we do not assume that $c_0,c_1$
are ${\mathcal C}^{1,1}$ (or semiconvex).

We provide a direct proof of uniqueness for the solution of the 
dislocation equation \eqref{dyseqn}; we recall that the existence 
of weak solutions is obtained in \cite{bclm07, bclmt08} and that, 
in our case, the weak solutions are classical solutions since the 
set $\{u=0\}$ has a zero Lebesgue measure by the result of
\cite{ley01} since $c[\chi]\geq 0$ for all $0\leq \chi\leq 1.$

\begin{theorem} \label{unicite-nl}
Suppose that $c_0,c_1$ satisfy {\bf (H3)} and that $u_0$ is a 
Lipschitz continuous function satisfying {\bf (H2)} and
such that \eqref{cond-infinity} holds. Then
\eqref{dyseqn}-\eqref{iddyseqn} 
has a unique (Lipschitz) continuous viscosity solution in $\R^N\times [0,T].$
\end{theorem}

\begin{proof}[Proof of Theorem~\ref{unicite-nl}]\ \\ 
\noindent{\it 1. Uniform bounds for the velocity.} 
By {\bf (H3)} and Lemma \ref{fini-speed}, the set $\{u(\cdot, t)\geq 0\}$ remains in a 
fixed ball $\bar{B}(0,R_0+\overline{c}T)$ of $\R^N.$
Then, for any subset $A$ of $B(0,R_0+\overline{c}T)$, 
$c[\car_A]$ satisfies {\bf (H1)} with constants which are uniform in
$A$.\\ 
\noindent{\it 2. $L^\infty$-estimate.} 
If $u_1, u_2$ are two solutions of \eqref{dyseqn}-\eqref{iddyseqn},
for $0< \tau\leq T,$ we set
\begin{eqnarray*}
\delta_\tau :=\sup_{\R^N\times [0,\tau]} \, |u_1 ( x,t) - u_2 (x,t)|.
\end{eqnarray*}
Since $u_0$ is Lipschitz continuous and $0\leq c[\car_{\{u_i \geq 0\}}]\leq \oc$
($i=1,2$),
for $\tau$ small enough, we have $\delta_\tau \leq \eta/2$ where 
$\eta$ is obtained by applying Theorem \ref{thm-borneinf}
to the $u_i$'s. By Lemma~\ref{traj}, we have
\begin{eqnarray}
\delta_\tau &\leq &
|D u_0|_\infty {\rm e}^{{C} \tau}  \int_0^\tau 
|(c[\car_{\{u_1 \geq 0\}}]-c[\car_{\{u_2 \geq 0\}}])(\cdot ,t)|_\infty dt
\nonumber \\
&\leq &
|D u_0|_\infty {\rm e}^{{C} \tau}  \int_0^\tau |c_0(\cdot ,t) * 
(\car_{\{u_1(\cdot ,t)\geq 0\}}-\car_{\{u_2(\cdot ,t)\geq 0\}})|_\infty dt
\nonumber \\
&\leq &
\oc \, |D u_0|_\infty {\rm e}^{{C} T} 
\int_0^\tau \int_{\R^N} |\car_{\{u_1\geq 0\}}-\car_{\{u_2\geq 0\}}|dx dt
\label{calc-dto}
\end{eqnarray}
by using the $L^\infty$-bound $|c_0|_\infty \leq \oc.$\\
\noindent{\it 3. $L^1$-estimate.} We have
$$
|\car_{\{u_1\geq 0\}}-\car_{\{u_2\geq 0\}}|
\leq \car_{\{-\delta_\tau\leq u_1 \leq 0 \}} + \car_{\{-\delta_\tau \leq u_2 \leq 0 \}}
\quad {\rm in} \ \R^N\times [0,\tau].
$$
Using Proposition~\ref{est-int} we get
$$
\int_0^\tau \int_{\R^N} |\car_{\{u_1\geq 0\}}-\car_{\{u_2\geq 0\}}|dx dt
 \leq \frac{2 \delta_\tau }{\etab \uc}\psi_\tau\; ,
$$
where we have set
$$
\psi_\tau = 
{\mathcal L}^N(\{x\, :\,  u_0 (x)\geq -\delta_\tau -\oc|Du_0|_\infty  \tau\})
- {\mathcal L}^N(\{x\, :\,u_0(x)\geq 0 \})\; .
$$ 

\noindent{\it 4. Uniqueness on $[0,\tau]$ for small $\tau.$} 
Using this information in \eqref{calc-dto} yields
\begin{eqnarray*}
\delta_\tau \leq \frac{2 \oc }{\etab \uc}|D u_0|_\infty {\rm e}^{{C} T} \psi_\tau\delta_\tau,
\end{eqnarray*}
namely
\begin{eqnarray*}
\delta_\tau \leq L\psi_\tau\delta_\tau,
\end{eqnarray*}
where ${L} = {L}(T,\uc,\oc, {C}, \etab,|D u_0|_\infty)$ is a
constant. 
Since the $0$-level set of $u_0$ has a zero Lebesgue-measure from assumption {\bf (H2)}, 
we have $\psi_\tau \to 0$ as $\tau~\to~0$. Therefore, for $\tau$ small enough, ${L}\psi_\tau  < 1$ and
necessarily $\delta_\tau=0.$ It follows $u_1=u_2$ on $\R^N \times [0,\tau].$

\noindent{\it 5. Uniqueness on the whole time interval.} 
Step~4 gives the uniqueness for small times but then we can consider
$$
\bar \tau = \sup\{\tau>0;\  u_1=u_2\ \hbox{on}\ \R^N \times[0,\tau]\}.
$$ 
In fact, by continuity of $u_1$ and $u_2$, $\bar \tau$ is a maximum.
If $\bar \tau<T$, then we can repeat 
the above proof from time $\bar \tau$ instead of time $0$. This is, 
in fact, rather straightforward since $u(\cdot,\bar \tau)$ satisfies 
the same properties as $u_0.$  Finally, $\bar \tau=T$ and
the proof is complete.
\end{proof}

\section{Model problem 2: a Fitzhugh-Nagumo type system}\label{sec:mod2}

We are now interested in the following system:  
\begin{eqnarray} \label{system}
  \begin{cases} 
u_t=\alpha (v)|Du| & {\rm in} \  \R^N \times (0,T),  \\  
v_t-\Delta v = g^+(v)\1_{\{ u\geq 0 \}}+g^-(v)(1-\1_{\{
  u\geq 0 \}}) & {\rm in} \  \R^N \times (0,T),\\
u(\cdot ,0)=u_0, \ v(\cdot ,0)=v_0  & {\rm in} \  \R^N,
\end{cases}    
\end{eqnarray}  
which is obtained as the asymptotics as $\e \to 0$ of the following Fitzhugh-Nagumo 
system arising in neural wave propagation or chemical kinetics (cf. \cite{ss96}):
\begin{equation}\label{system-eps}
\left\lbrace
\begin{aligned}
u^{\e}_t-\e \Delta u^{\e}&=\frac{1}{\e} f(u^{\e},v^{\e}) & {\rm in} \  \R^N \times (0,T), \\
v^{\e}_t-\Delta v^{\e}&=g(u^{\e},v^{\e}) & {\rm in} \  \R^N \times (0,T),
\end{aligned}
\right.
\end{equation}
where
$$
\left\lbrace
\begin{aligned}
f(u,v)&= u(1-u)(u-a) -v & (0<a<1), \\
g(u,v)&=u-\gamma v & (\gamma >0).
\end{aligned}
\right.
$$
The functions $\alpha$, $g^+$ and $g^-$ appearing in \eqref{system} are Lipschitz continuous 
functions on $\R$ associated with $f$ and $g$. The functions $g^-$ and $g^+$ are bounded 
and satisfy $g^- \leq g^+$ in $\R.$ The initial datum $v_0$ is bounded and of
class ${\mathcal C}^1$ in $\R^N$ with $|Dv_0|_{\infty} <+\infty$. \\

System \eqref{system} corresponds to a front
$\Gamma(t)=\{u(\cdot,t)=0\}$ 
moving with normal velocity $\a (v)$, 
the function $v$ being itself the solution of an interface 
reaction-diffusion equation depending on the regions separated 
by $\Gamma(t)$. The $u$-equation in \eqref{system} can be written 
as  \eqref{dyseqn}-\eqref{iddyseqn} although the dependence of $c$ 
in $\car_{\{u(\cdot ,t)\geq 0\}}$ is less explicit than in the first 
model case. More precisely, for 
$\chi \in L^{\infty}(\R^N \times [0,T],[0,1])$, let $v$ be the solution of
\begin{equation}\label{heat-eq}
\left\{
\begin{array}{ll}
v_t-\Delta v = g^+(v)\chi+g^-(v)(1-\chi) & {\rm in} \,  \R^N \times [0,T], \\
v(\cdot ,0)=v_0  & {\rm in} \, \R^N\;.
\end{array} \right.
\end{equation}  
Then Problem \eqref{system} reduces to \eqref{dyseqn}-\eqref{iddyseqn} with $c[\chi](x,t) = \alpha (v(x,t))$.

Under the additional assumption that $\alpha> 0$ in $\R,$ 
we prove uniqueness of solutions to the system \eqref{system} 
(or equivalently \eqref{dyseqn}).
We suppose \\

\noindent{\bf (H4)} $v_0$ is bounded and ${\mathcal C}^1$, $g^-, g^+$ are Lipschitz continuous with
\begin{eqnarray*}
|Dv_0|_\infty<+\infty \ \ \ {\rm and} \ \ \  
\underline{g}\leq g^-(r) \leq g^+(r) \leq \overline{g}
\quad {\rm for \ all} \ r\in\R.
\end{eqnarray*}

\noindent{\bf (H5)}
$\alpha$ is Lipschitz continuous and there exists 
$\underline{c},\overline{c}, C>0$ such that,
 for all $r,r'\in\R,$
\begin{eqnarray*}
&\underline{c}\leq \alpha(r) \leq \overline{c},&\\
&|\alpha(r)-\alpha(r')|\leq C|r-r'|.&
\end{eqnarray*}

\noindent{\bf (H6)} $u_0$ is Lipschitz continuous 
and satisfies \eqref{cond-infinity} with $K_0:=\{u_0\geq 0\}$ 
which is the closure of a non empty bounded open subset of $\R^N$ with ${\mathcal C}^2$ boundary.\\

\begin{theorem}\label{uniqueness-fn}
Under assumptions {\bf (H2)}, {\bf (H4)}, {\bf (H5)}, {\bf (H6)}, system \eqref{system} has a 
unique solution.  
\end{theorem}

We recall that the existence of weak solutions 
is obtained in \cite{ggi92, ss96}. Moreover, since $\alpha>0$ in $\R,$ 
weak solutions are classical thanks to the results of \cite{ley01}. 
Before giving the uniqueness proof, we start by a
preliminary on the inhomogeneous heat equation.

\subsection{Classical estimates for the inhomogeneous heat equation.}
\label{est-he}

We first gather some regularity results for the solutions of the heat equation (\ref{heat-eq}). The explicit 
resolution of the heat equation \eqref{heat-eq} 
shows that for any $(x,t) \in \R^N \times [0,T]$,
$$
\begin{aligned}
v(x,t)=\int_{\R^N} &G(x-y,t)  \, v_0(y) \, dy \\ 
&+ \int_0^t \int_{\R^N} G(x-y,t-s)  \, [g^+(v)\chi +g^-(v)(1-\chi )](y,s) \, dy ds,
\end{aligned}
$$
where $G$ is the Green function defined by
\begin{equation} \label{green}
G(y,s)=\frac{1}{(4\pi s)^{N/2}} e^{-\frac{|y|^2}{4s}}.  
\end{equation}

It is then easy to obtain the following lemma.

\begin{lemma}\label{reg-velocity}  
Assume that {\bf (H4)} holds.
For $\chi \in L^{\infty}(\R^N \times [0,T];[0,1])$, let $v$ be the unique
solution of \eqref{heat-eq}. Set $\gamma=\max \{|\underline{g}|,|\overline{g}|\}.$ 
Then there exists a constant $k_N$ depending only on $N$ such that
\begin{itemize}

\item[(i)] $v$ is uniformly bounded: for all $(x,t) \in \R^N \times [0,T],$
$$
|v(x,t)| \leq |v_0|_{\infty} + \gamma t.
$$ 

 \item[(ii)] $v$ is continuous on $\R^N \times [0,T]$.

 \item[(iii)]  For any $t \in [0,T]$, $v(\cdot,t)$ is of 
class ${\mathcal C}^1$ in $\R^N$.

 \item[(iv)] For all $t \in [0,T]$, $x,y \in \R^N$,  
$$|v(x,t)-v(y,t)|\leq (\, |Dv_0|_{\infty} +\gamma k_N \,\sqrt{t}) \,|x-y|.$$

 \item[(v)]  For all $0\leq s \leq t \leq T,$ $x \in \R^N$,  
$$
|v(x,t)-v(x,s)|\leq k_N
(|Dv_0|_{\infty}+ \gamma k_N 
\,\sqrt{s}) \,\sqrt{t-s} + \gamma (t-s).
$$

\end{itemize}
\end{lemma}  
In particular the velocity $c[\chi]$ (given here by $\alpha(v)$) 
is bounded, continuous on $\R^N \times [0,T]$ and Lipschitz continuous in space, 
uniformly with respect to $\chi$. It follows that \eqref{eq-gene2} has
a unique continuous (classical) viscosity solution for all 
$\chi \in L^{\infty}(\R^N \times [0,T];[0,1]).$

\subsection{Proof of Theorem \ref{uniqueness-fn}}

\noindent{\it 1. Properties of the velocity.} 
As explained above, for any measurable subset
$A$ of $\R^N,$ the velocity $c[\car_{A}]$ in
\eqref{dyseqn} satisfies {\bf (H1)} with constants which are uniform in
$A$: for all $x,x'\in\R^N,$ $t\in [0,T],$
\begin{eqnarray*}
&\uc \leq c[\car_{A}]\leq \oc& \\
& |c[\car_{A}](x,t)-c[\car_{A}](x',t)|\leq \tilde{C}|x-x'|,&
\end{eqnarray*}
with $\tilde{C}:= C(|Dv_0|_\infty +\gamma k_N\sqrt{T}).$
By Lemma \ref{fini-speed}, it follows that the set $\{u(\cdot, t)\geq~0\}$ remains in a 
fixed ball $\bar{B}(0,R_0+\overline{c}T)$ of $\R^N.$\\ 

\noindent{\it 2. First estimate (eikonal equation).} 
We start as in the proof of Theorem \ref{unicite-nl}.
Let $u_1, u_2$ be two solutions of \eqref{dyseqn}
and $v_1, v_2$ be the solutions of \eqref{heat-eq} associated with 
$u_1,u_2$ respectively.
For $0\leq \tau\leq T,$ we set
\begin{eqnarray*}
\delta_\tau :=\sup_{\R^N\times [0,\tau]} \, |u_1 ( x,t) - u_2 (x,t)|
\end{eqnarray*}
and we choose $\tau$ small enough in order that
$\delta_\tau < \eta/2$ where 
$\eta$ is given by applying Theorem \ref{thm-borneinf}
to the $u_i$'s. By Lemma~\ref{traj}, we have
\begin{eqnarray}
\delta_\tau &\leq &
|D u_0|_\infty {\rm e}^{\tilde{C} \tau}  \int_0^\tau 
|(c[\car_{\{u_1 \geq 0\}}]-c[\car_{\{u_2 \geq 0\}}])(\cdot ,t)|_\infty dt
\nonumber \\
&\leq &
|D u_0|_\infty {\rm e}^{\tilde{C} \tau}  \int_0^\tau 
|(\alpha(v_1)-\alpha(v_2))(\cdot ,t)|_\infty dt
\nonumber \\
&\leq &
C |D u_0|_\infty {\rm e}^{\tilde{C} T} 
 \int_0^\tau |(v_1-v_2)(\cdot ,t)|_\infty dt.
\label{calc-hea}
\end{eqnarray}
It remains to estimate $|(v_1-v_2)(\cdot ,t)|_\infty.$\\

\noindent{\it 3. Second Estimate (heat equation).} 
The function $v=v_1-v_2$ solves
\begin{eqnarray*}
v_t-\Delta v &=& 
(\1_{\{ u_1 \geq 0 \}}-\1_{\{ u_2 \geq 0 \}}) (g^+(v_1)-g^-(v_1)) \\ 
&&+  \1_{\{ u_2 \geq 0 \}} (g^+(v_1)-g^+(v_2)) - \1_{\{ u_2 \geq 0 \}} (g^-(v_1)-g^-(v_2))\\
&& +(g^-(v_1)-g^-(v_2))
\end{eqnarray*}  
in $\R^N\times [0,T].$
Since $g^+$ and $g^-$ are Lipschitz continuous, say with Lipschitz constant $M$, we have
$$
| \1_{\{ u_2 \geq 0 \}} (g^+(v_1)-g^+(v_2)) - \1_{\{ u_2 \geq 0 \}} (g^-(v_1)-g^-(v_2)) 
+(g^-(v_1)-g^-(v_2))| \leq 3M |v|.
$$
Moreover
$$
|\1_{\{ u_1 \geq 0 \}}-\1_{\{ u_2 \geq 0 \}}| \, |g^+(v_1)-g^-(v_1)|
\leq (\overline{g}-\underline{g})|\1_{\{ u_1 \geq 0 \}}-\1_{\{ u_2 \geq 0 \}}|,
$$
by {\bf (H4)}.
This implies that both $v$ and $-v$ are viscosity subsolutions of 
$$
w_t-\Delta w -3M|w| = 
(\overline{g}-\underline{g})|\1_{\{ u_1 \geq 0 \}}-\1_{\{ u_2 \geq 0
  \}}|
\ \ \ {\rm in} \ \R^N\times [0,T],
$$
whence $|v|=\max\{ v,-v\}$ is also a subsolution as the maximum of two
subsolutions. Therefore we have
$$
|v|_t-\Delta |v| -3M|v| 
\leq (\overline{g}-\underline{g})|\1_{\{ u_1 \geq 0 \}}-\1_{\{ u_2\geq 0 \}}|
\ \ \ {\rm in} \ \R^N\times [0,T].
$$
In particular the function $w:(x,t)\mapsto e^{-3Mt}|v(x,t)|$ satisfies 
$$
w_t - \Delta w \leq (\overline{g}-\underline{g})\, 
{\rm e}^{-3Mt}|\1_{\{ u_1 \geq 0 \}}-\1_{\{ u_2 \geq 0 \}}|
\ \ \ {\rm in} \ \R^N\times [0,T].
$$
By the comparison principle, since $w(\cdot ,0)=0,$ 
we have for any $(x,t) \in \R^N \times [0,\tau],$
$$
w(x,t) \leq \int_0^t \int_{\R^N} G(x-y,t-s)  \, 
(\overline{g}-\underline{g})\, {\rm e}^{-3Ms} \,
|\1_{\{ u_1 \geq 0 \}}-\1_{\{ u_2 \geq 0 \}}|(y,s) \, dy ds.
$$
Using the definition of $\delta_\tau$, we have
$$
|\1_{\{ u_1 \geq 0 \}}-\1_{\{ u_2 \geq 0 \}}|(y,s) \leq 
\car_{\{-\delta_\tau \leq u_1 < 0\}} 
+ \car_{\{-\delta_\tau\leq u_2 < 0\}}\;.
$$
This implies that 
for any $(x,t) \in \R^N \times [0,\tau]$,
\begin{eqnarray}
&& |v_1(x,t)-v_2(x,t)| \label{form129}\\ 
\nonumber 
&\leq &
(\overline{g}-\underline{g})\, {\rm e}^{3MT} 
\int_0^{t} \int_{\R^N} G(x-y,t-s) \, 
\left( \car_{\{-\delta_\tau \leq u_1 < 0\}} 
+ \car_{\{-\delta_\tau \leq u_2 < 0\}}\right) dy ds.
\end{eqnarray}
  
For simplicity, we set $B=\bar{B}(0,1)$ and
\begin{eqnarray*}
K_i(t)=\{u_i(\cdot,t) \geq 0\} \ \ \ {\rm for} \ i=1,2.
\end{eqnarray*}  

\noindent{\it 4. We claim that $\{-\delta_\tau \leq u_i(\cdot ,t) < 0\}
\subset (K_i(t)+2\delta_\tau B/\bar{\eta})\setminus K_i(t) $
where $\bar{\eta}$ is given by \eqref{bornepresfront}.} 
Indeed let $x\in\R^N$ be such that $-\delta_\tau \leq u_i(x,t) < 0.$
Since we chose $\delta_\tau$ small enough in Step 2, 
\eqref{bornepresfront} holds and Lemma \ref{visc-increas}
implies that there exists 
$y\in \bar{B}(x,2\delta_\tau/\bar{\eta})$ such that
$u_i(y,t)\geq u_i(x,t)+\delta_\tau\geq 0.$ This proves the
claim.\\

\noindent{\it 5. Use of an interior cone property 
for the $K_i(t)$'s.}
Note that $\{-\delta_\tau \leq u_i(\cdot ,t) \leq 0\}
\setminus \{-\delta_\tau \leq u_i(\cdot ,t) < 0\}$
has a 0 Lebesgue measure since the velocity is nonnegative
(cf. \cite{ley01}).
Then, from \eqref{form129} and Step 4, we obtain
\begin{eqnarray}
&& |v_1(x,t)-v_2(x,t)| \label{abf1}\\
\nonumber 
&\leq &
(\overline{g}-\underline{g})\, {\rm e}^{3MT} 
\int_0^{t} \int_{\R^N} G(x-y,t-s) \, 
 (\1_{E_1(t)}(y) + 
\1_{E_2(t)}(y))\, dy ds
\end{eqnarray} 
where $E_i(t)=(K_i(t)+2\delta_\tau B/\bar{\eta})\setminus K_i(t)$ for $i=1, 2$. 

We are now going to use the fact that the sets $K_1(t)=\{ u_1(\cdot,t)
\geq 0 \}$ and $K_2(t)=\{ u_2(\cdot,t) \geq 0 \}$ 
have the {interior cone property} (see Definition \ref{def-int-cone-prop})
for all $t \in [0,T]$, for some parameters $\rho$ and $\theta$ independent of $t$: 
\begin{lemma} 
There exist $\rho$ and $\theta$ depending only on the data 
($\alpha$, $u_0$, $v_0$, $g^+$ and $g^-$) such 
that $0<\rho<\theta<1$ and $K_i(t)$ has the interior cone property 
of parameters $\rho$ and $\theta$ for all $t \in [0,T]$. 
\end{lemma} 
\noindent This lemma is an application of Theorem \ref{propagation-cone} below
(see Section \ref{ppg-cone}), 
the assumptions of which are satisfied for $u_1,u_2$ because of Step 1.
It follows that we can use the following lemma which is proved Section \ref{green-kern}:
\begin{lemma} \label{th-estimate-greenkernel}  
Let $\{ K(t) \}_{t \in [0,T]} \subset \bar{B}(0,R) \times [0,T]$ be a 
bounded family of compact subsets of $\R^N$ having the interior cone 
property of parameters $\rho$ and $\theta$ with $0<\rho<\theta<1$ 
and $R>0,$ 
and let us set, for any $x \in \R^N$, $t \in [0,T]$ and $r \geq 0$,  
$$  
\phi(x,t,r)=\int_0^t \int_{\R^N} G(x-y,t-s) \, \1_{K(s)+rB}(y) \, dy ds.  
$$  
Then for any $r_0 >0$ and $0\leq \tau <1,$ there exists a constant 
$\Lambda_0=\Lambda_0(\tau,N,R,r_0,\rho,\theta/\rho)$ 
such that for any $x \in \R^N$, $t \in [0,\tau]$ and $r \in [0,r_0]$,  
$$  
|\phi(x,t,r)-\phi(x,t,0)| \leq \Lambda_0 \, r.  
$$   
\end{lemma}  
\noindent We apply this lemma to the $K_i(t)$'s which verify the assumptions 
with $R=R_0+\oc T$ by Step 1 and since we can assume that $\tau <1$. From 
\eqref{calc-hea} and \eqref{abf1},
we finally obtain that 
\begin{equation*}  
\delta_\tau \leq L\tau \delta_\tau
\end{equation*}  
where $L=L(T,C,  \tilde{C}, |Du_0|_\infty, \underline{g}, \overline{g},
M, \bar{\eta}, \Lambda_0).$
Choosing $\tau$ such that $L\tau <1,$ we obtain $\delta_\tau=0.$
We conclude as in the proof of Theorem \ref{unicite-nl}. 
\hfill $\Box$\\

\subsection{Proof of Lemma \ref{th-estimate-greenkernel}}  \label{green-kern}

\noindent For any $x \in \R^N$, $t \in [0,\tau]$ and $r \geq 0$,  
$$  
\phi(x,t,r)-\phi(x,t,0)=\int_0^t \int_{\R^N} G(x-y,t-s) \, 
\left( \1_{K(s)+rB} - \1_{K(s)} \right) (y)\, dy ds.  
$$  
Let $\overline{d}_{K(s)}$ denote the signed distance function to $K(s)$, namely  
\begin{equation*}  
\left\lbrace  
\begin{array}{cc}  
\overline{d}_{K(s)}(x)=d_{K(s)}(x) & \text{if} \ x \notin K(s), \\  
\overline{d}_{K(s)}(x)= -d_{\partial K(s)}(x) & \text{if} \ x \in K(s),  
\end{array}
\right.  
\end{equation*}  
where, for any $A\subset \R^N,$ $d_A$ is the usual distance to $A.$
Then $\1_{K(s)+rB} - \1_{K(s)}=\1_{\{0 < \overline{d}_{K(s)} \leq r\}}$, so that  
$$  
\phi(x,t,r)-\phi(x,t,0)=\int_0^t 
\int_{\{0 < \overline{d}_{K(s)} \leq r\}} G(x-y,t-s) \, dy ds.  
$$  
Since $\overline{d}_{K(s)}$ is Lipschitz continuous with 
$|D \overline{d}_{K(s)}|=1$ 
almost everywhere, the coarea formula (see \cite{eg92}) shows that  
$$  
\begin{aligned}  
\phi(x,t,r)-\phi(x,t,0)  
&=\int_0^t \int_0^r \int_{\{ \overline{d}_{K(s)} =\sigma\}} 
G(x-y,t-s) \, d\Hc^{N-1}(y) d\sigma ds\\  
&=\int_0^t \int_0^r d\sigma \int_{\{ \overline{d}_{K(s)} =\sigma\}} 
\frac{1}{(4\pi (t-s))^{N/2}} e^{-\frac{|x-y|^2}{4(t-s)}} \, d\Hc^{N-1}(y) ds.  
\end{aligned}  
$$  
The change of variable $z=\frac{x-y}{\sqrt{t-s}}$ in this last integral yields  
$$  
\phi(x,t,r)-\phi(x,t,0)=\frac{1}{(4 \pi)^{N/2}} \int_0^r d\sigma
\int_0^t 
\frac{1}{\sqrt{t-s}} \int_{\zeta_{s,\sigma}}  e^{-\frac{|z|^2}{4}} \, d\Hc^{N-1}(z) ds,  
$$  
where we have set $$\zeta_{s,\sigma}= \left\lbrace
\frac{y-x}{\sqrt{t-s}};\; 
\overline{d}_{K(s)}(y) =\sigma \right\rbrace.$$ For some $R(s)$ 
to be precised later, we split 
$\int_{\zeta_{s,\sigma}}  e^{-\frac{|z|^2}{4}} \, d\Hc^{N-1}(z)$ 
in two parts, one in $B_{R(s)}=\bar{B}(0,R(s))$, and one in $B_{R(s)}^c$. 
First, for any $s \in [0,t)$ and $\sigma >0$,  
\begin{eqnarray*}    
 && \hspace*{-3cm} \int_{\zeta_{s,\sigma} \cap B_{R(s)}}  e^{-\frac{|z|^2}{4}} \,  d\Hc^{N-1}(z) \\
 &\leq & \Hc^{N-1}(\zeta_{s,\sigma} \cap B_{R(s)}) \\
 & \leq & \Lambda(N,\rho,\theta / \rho) {\mathcal{L}}^N(B(0,1)) (R(s)+\rho/4)^N\\  
 &\leq & \Lambda(N,\rho,\theta / \rho) \, {\mathcal{L}}^N(B(0,1)) (R(s)+1)^N  
\end{eqnarray*}  
where $\Lambda(N,\rho,\theta / \rho)$ is the constant given by Theorem
\ref{interior-cone}. 
Indeed, for any $s \in [0,t)$ and $\sigma >0$,  
$$  
\zeta_{s,\sigma}=\partial \left\lbrace \frac{y-x}{\sqrt{t-s}};\; 
\overline{d}_{K(s)}(y) <\sigma \right\rbrace,  
$$  
and these sets inherit the interior cone property of parameters 
greater than $\rho/\max(\sqrt{\tau},1)=\rho$ and
$\theta/\max(\sqrt{\tau},1)=\theta$ 
from $K(s)$ (we recall that we have assumed $\tau<1$). Besides  
\begin{eqnarray*}  
&& \hspace*{-1.5cm} \int_{\zeta_{s,\sigma} \cap B_{R(s)}^c}  e^{-\frac{|z|^2}{4}} \,
d\Hc^{N-1}(z) \\
&\leq & e^{-\frac{R(s)^2}{4}} \Hc^{N-1}(\zeta_{s,\sigma}) \\ 
&\leq &
e^{-\frac{R(s)^2}{4}} \frac{1}{(t-s)^{\frac{N-1}{2}}}
\Hc^{N-1}(\left\lbrace \overline{d}_{K(s)} =\sigma \right\rbrace) \\ 
&\leq & e^{-\frac{R(s)^2}{4}} \frac{1}{(t-s)^{\frac{N-1}{2}}} \, 
\Lambda(N,\rho,\theta / \rho) \, {\mathcal{L}}^N(B(0,1)) (R+r_0+\rho/4)^N \\  
& \leq & e^{-\frac{R(s)^2}{4}} \frac{1}{(t-s)^{\frac{N-1}{2}}} \, 
\Lambda(N,\rho,\theta / \rho) \, {\mathcal{L}}^N(B(0,1)) (R+r_0+1)^N,  
\end{eqnarray*}  
because $\left\lbrace \overline{d}_{K(s)} \leq \sigma \right\rbrace 
\subset B_{R+r_0}$ for any $s \in [0,\tau]$ and $r \in [0,r_0]$. 
This last estimate also comes from Theorem \ref{interior-cone} 
for the same reason as above. Thus we have proved the existence 
of a constant   
$$  
\Lambda_1=\Lambda_1(N,R,r_0,\rho,\theta / \rho)=\frac{1}{(4 \pi)^{N/2}} 
\Lambda(N,\rho,\theta / \rho) \, {\mathcal{L}}^N(B(0,1)) (R+r_0+1)^N
$$   
such that for any $x \in \R^N$, $t \in [0,\tau]$ and $r \in [0,r_0]$,  
\begin{equation}\label{majphi}
| \phi(x,t,r)-\phi(x,t,0)| \leq \Lambda_1 \, r \, \int_0^t 
\frac{1}{\sqrt{t-s}} \left( (R(s)+1)^N 
+  \frac{e^{-\frac{R(s)^2}{4}}}{(t-s)^{\frac{N-1}{2}}} \right) \,  ds.  
\end{equation}
Choosing $R(s)=\sqrt{-2(N-1) {\rm log}(t-s)}$, 
so that $e^{-\frac{R(s)^2}{4}}=(t-s)^{\frac{N-1}{2}}$, 
we can estimate the right-hand side of (\ref{majphi})  as follows:
\begin{equation*}  
\begin{aligned}  
 & \int_{0}^t \frac{1}{\sqrt{t-s}} 
\left( (R(s)+1)^N +
\frac{e^{-\frac{R(s)^2}{4}}}{(t-s)^{\frac{N-1}{2}}} \right) \,  ds \\
 & \leq  \int_{0}^{1} 
\frac{(|2(N-1){\rm log}(u)|^{1/2}+1)^N + 1}{\sqrt{u}} \, du =:I(N)\;.  
\end{aligned} \end{equation*} 
We deduce the existence of the constant  
$$\Lambda_0=\Lambda_0(\tau,N,R,r_0,\rho,\theta / \rho)=\Lambda_1  I(N)
$$   
such that for any $x \in \R^N$, $t \in [0,\tau]$ and $r \in [0,r_0]$,   
$$  
| \phi(x,t,r)-\phi(x,t,0)| \leq \Lambda_0 \, r.  
$$  
\ \hfill$\Box$

\section{Eikonal equation, interior cone property and perimeter estimates}\label{auxil}

\subsection{Some results on the classical eikonal equation}

In this section, we collect several properties of the eikonal
equation \eqref{chjb}.

We first recall the
\begin{theorem} [\cite{ley01}] \label{thm-borneinf} \

\begin{enumerate}
\item[(i)] Under assumption {\bf (H1)}, equation \eqref{chjb} 
has a unique continuous viscosity solution $u.$ If $u_0$ is Lipschitz
continuous, then $u$ is Lipschitz continuous and, 
for almost all $x\in \R^N,$ $t\in [0,T],$
$$
|Du(x,t)|\leq {\rm e}^{{C}T}|D u_0|_{\infty}\, , \ \  \ \quad  
|u_t (x,t)|\leq \oc{\rm e}^{{C}T}|D u_0|_{\infty}\; .
$$
\item[(ii)] Assume that $u_0$ is Lipschitz continuous 
and that {\bf (H1)} and {\bf (H2)} hold.
Then there exist 
$\gamma =\gamma ({C},\oc,\eta_0) >0, \eta =\eta ({C},\oc,\eta_0) >0$
such that the viscosity solution $u$ of \eqref{chjb}  
satisfies in the viscosity sense
\begin{eqnarray}\label{borneinf}
-|u(x,t)|-\frac{e^{\gamma t}}{4}|D u(x,t)|^2+\eta \leq 0 \ { in} \ \R^N\times [0,T] \; .
\end{eqnarray}
\end{enumerate}
\end{theorem}
We refer the reader to \cite{ley01} for the proof of this result. 
Let us mention that {\bf (H1)} implies that $p\in\R^N \mapsto c(x,t)|p|$
is convex for every $(x,t)\in\R^N\times [0,T]$ which is a key
assumption to prove (ii).
We remark that, in (ii), $u$ is Lipschitz continuous
because the assumptions of (i) are satisfied. Therefore $u$ is differentiable a.e. in
$\R^N\times [0,T]$ and \eqref{borneinf} holds a.e. in $\R^N\times [0,T].$
Part (ii)  gives a lower-bound gradient estimate for $u$ near the front
$\{ (x,t)\in \R^N\times [0,T] : u(x,t)=0\}.$ Indeed, if $|u(x,t)|< \eta /2,$ then 
\begin{eqnarray} \label{bornepresfront}
-|D u(x,t)|\leq -\sqrt{2\eta} e^{-\gamma t/2}:= - \etab <0
\ {\rm in} \ \R^N\times [0,T] 
\end{eqnarray}
in the viscosity sense (and almost everywhere in $\R^N\times [0,T]$).

We continue by giving an upper-bound for the difference of 
two solutions with different velocities $c_i.$
\begin{lemma}[\cite{bl06}] \label{traj} 
For $i=1,2$, let $u_i \in {\mathcal C}^0(\R^N\times [0,T])$ be a solution of
\begin{eqnarray*}
\left\{
\begin{array}{cc}
(u_i)_t = c_i(x,t)|D u_i | & { in} \ \R^N\times [0,T],\\[2mm]
u_i (x,0)=u_0(x) & { in} \ \R^N,
\end{array}
\right.
\end{eqnarray*}
where $c_i$ satisfies {\bf (H1)} and $u_0$ is Lipschitz continuous. Then, for any $t\in [0,T],$
$$
 |(u_1-u_2)(\cdot ,t)|_\infty
\leq |D u_0|_\infty {\rm e}^{{C} t}  \int_0^t 
|(c_1-c_2)(\cdot ,s)|_\infty  ds.
$$
\end{lemma}

Finite speed of propagation implies
a uniform bound for compact fronts governed by
eikonal equations:

\begin{lemma} [\cite{bl06}] \label{fini-speed}
Suppose that {\bf (H1)} holds and that $u_0$ is Lispchitz continuous
and satisfies \eqref{cond-infinity}.
Let $u$ be the viscosity solution of \eqref{chjb} with initial
condition $u_0.$ Then, for all $t\in [0,T],$
\begin{eqnarray*} 
\{u(\cdot ,t)\geq 0\} \subset \bar{B}(0,R_0+\overline{c}t).
\end{eqnarray*}
\end{lemma}

\begin{lemma}[\cite{bl06}] (viscosity increase principle) 
\label{visc-increas}
Let $w\in {\mathcal C}^0(\R^N)$ satisfying {\bf (H2)} and $\delta <\eta_0/2.$
If $x\in \{-\delta \leq w\leq \delta\},$ then
\begin{eqnarray*} 
\mathop{\rm sup}_{\bar{B}(x,2\delta/\eta_0)} w \geq w(x)+\delta.
\end{eqnarray*}
\end{lemma}

We refer the reader to \cite{bl06} for the proofs of these results.

\subsection{Estimates on the measure of level-sets for
solutions of  \eqref{chjb}.}

Now we turn to the key estimates on the measure of small level-sets of
the solution of the Eikonal equation \eqref{chjb}.
For every $-\eta/2 \leq a <b \leq \eta/2$, we consider the function 
$\varphi : \R\to \R^+$ , depending on $a$ and $b$ such that
$\varphi =0$ on $(-\infty ,a),$ $\varphi ' (t) = (b-a)^{-1}$ in $(a,b)$ and
$\varphi =1$ on $[b , +\infty).$ In fact, $\varphi$ is chosen in such a 
way that $(b-a) \varphi'$ is the indicator function of $[a,b].$
We omit to write the dependence of $\varphi$ with respect to 
$a,b$ for the sake of simplicity of notations.

\begin{proposition} \label{est-int}
Assume {\bf (H1)}, {\bf (H2)} and suppose that $\{ u_0\geq 0\}$ 
is a compact subset of $\R^N$. Let $-{\eta}/2 \leq a <b \leq {\eta}/2$ 
where ${\eta}$ is defined in \eqref{borneinf} and 
let $u$ be the unique Lipschitz continuous viscosity solution of \eqref{chjb}.
Then, for any $0 < \tau \leq T$
\begin{equation}\label{est-base}
\int_0^\tau \int_{\R^N}\, \car_{\{a \leq u \leq b\}} dxdt \leq \frac{b-a}{\etab \uc}\int_{\R^N}\,
\left[ \varphi(u(x,\tau)) - \varphi(u(x,0)) \right]dx,
\end{equation}
where $\bar{\eta}$ is defined in \eqref{bornepresfront}.
It follows
\begin{eqnarray}\label{est-base1}
 \int_0^\tau \int_{\R^N}\, \!\! \car_{\{a \leq u \leq b\}} dxdt
\leq 
\frac{b\!-\!a}{\etab \uc} \left[{\mathcal L}^N\!\left(\{u(\cdot,\tau)\geq a\}\right) 
- {\mathcal L}^N\!\left(\{u(\cdot,0)\geq b\}\right)\right]\! dx,
\end{eqnarray}
and
\begin{eqnarray}\label{est-base2}
&& \int_0^\tau \int_{\R^N}\, \car_{\{a \leq u \leq b\}} dxdt \\
&& \hspace*{1.5cm} \leq  \frac{b-a}{\etab \uc} \left[
{\mathcal L}^N\left(\{u(\cdot,0)\geq a-\oc|Du_0|_\infty
\tau\}\right) 
- {\mathcal L}^N\left(\{u(\cdot,0)\geq b\}\right) \right]. \nonumber
\end{eqnarray}
\end{proposition}

\begin{remark} The above Proposition is related with results obtained by the fourth author in \cite{monteillet07}
for the eikonal equation with a changing sign velocity.
\end{remark}

\begin{proof}[Proof of Proposition~\ref{est-int}]
By the definition of $\varphi$
$$
\int_0^\tau \int_{\R^N}\, \car_{\{a \leq u \leq b\}} dxdt 
= \int_0^\tau \int_{\R^N}\, (b-a)\varphi'(u (x,t)) dxdt\; .
$$
Using the fact that $-{\eta}/2 \leq a <b \leq {\eta}/2$ 
and the definition of $\bar{\eta}$ in \eqref{bornepresfront}, 
we can estimate the right-hand side by
$$ 
 \int_0^\tau \int_{\R^N}\, (b-a)\varphi'(u (x,t))
\frac{c(x,t)}{\uc}\frac{|Du|}{\etab} dxdt \; ,
$$
since $\uc \leq c$ on $\R^N \times (0,T)$ and $|Du| \geq \etab$ 
on the set $\{|u|\leq \eta/2\}$. Therefore, by the equation, we have the following equality
$$
\frac{b-a}{\uc\etab}
\int_0^\tau \int_{\R^N}\, \varphi'(u (x,t))c(x,t)|Du| dxdt 
= \frac{(b-a)}{\uc\etab}\int_0^\tau \int_{\R^N}\, \left(\varphi(u
(x,t))\right)_t dxdt \; ,
$$
and \eqref{est-base} follows by applying Fubini's Theorem and  integrating.
Inequality \eqref{est-base1}  follows easily by 
taking into account the form of $\varphi.$
To deduce \eqref{est-base2}, it is sufficient to note that, 
since $u_0+ \oc|Du_0|_\infty t$ 
is a supersolution of \eqref{chjb}, we have, by comparison, 
$u(x,t) \leq u_0 (x) + \oc|Du_0|_\infty t$
in $\R^N \times (0,T)$.
\end{proof}

\subsection{Estimate of the perimeter of sets with the interior cone property}  

\begin{definition} \label{def-int-cone-prop}
Let $K$ be a compact subset of $\R^N$. We say that $K$ has the interior cone property 
of parameters $\rho$ and $\theta$ if $0<\rho < \theta$ and if, for any   
$x \in \partial K$, there exists some  $\nu \in \Sc^{N-1}$ such that the set
\begin{eqnarray*} 
\begin{array}{ccl}
 {\mathcal C}_{\nu,x}^{\rho,\theta} & \!\! := & \!\! x +
     [0,\theta]\bar{B}(\nu,\rho / \theta) \\
 & \!\! = &\!\! \{x+\lambda\nu+\lambda\frac{\rho}{\theta}\xi\, : \, 
\lambda\in [0,\theta], \, \xi\in\bar{B}(0,1)\}  
\end{array}
\end{eqnarray*}
is contained in $K$.
\end{definition}  

\begin{theorem} \label{interior-cone}  
Let $K$ be a compact subset of $\R^N$ having the interior cone property of parameters 
$\rho$ and $\theta$.  
Then there exists a positive constant ${\Lambda}=\Lambda(N,\rho,\theta/\rho)$ such that for all $R >0$,  
\begin{eqnarray}\label{est-perim} 
\Hc^{N-1}(\partial K \cap \bar{B}(0,R)) \leq \Lambda\, \Lc^{N}(K \cap \bar{B}(0,R+\rho /4)).  
\end{eqnarray} 
\end{theorem}  

\begin{figure}[ht]  
\begin{center}  
\epsfig{file=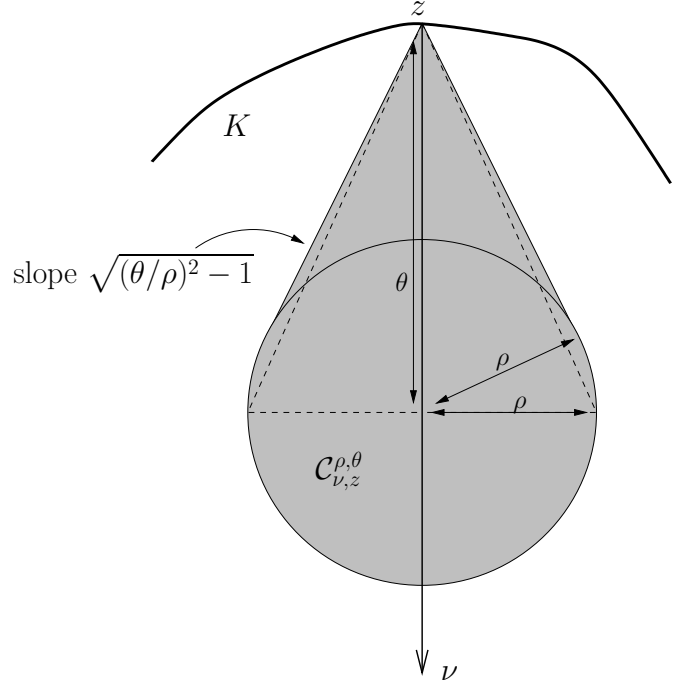, width=9cm}   
\end{center}  
\caption{\label{des-def-cone}  
{\it ${\mathcal C}_{\nu, z}^{\rho,\theta}$: interior cone at $z$ of parameters 
  $\rho,\theta$ and axis $\nu.$}}  
\end{figure}  
  
\begin{figure}[ht]  
\begin{center}  
\epsfig{file=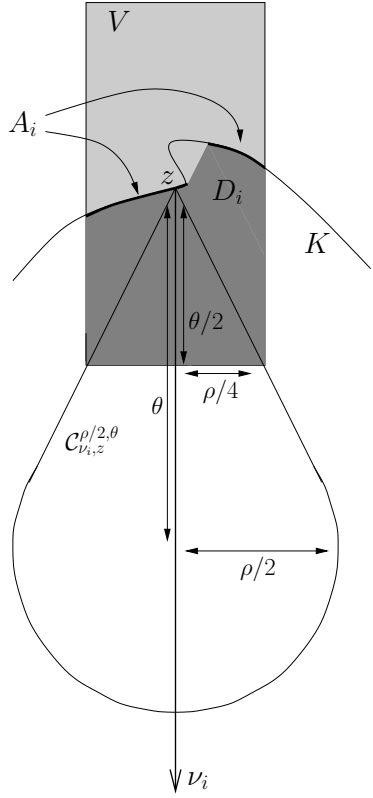, width=5cm}   
\end{center}  
\caption{\label{dess-preuve-i}  
{\it Illustration of the proof of Theorem \ref{interior-cone}.}}  
\end{figure}  

\begin{proof} \ \\ 
 \noindent \textit{1. Restriction to a finite number of axes for the interior  cones.}  
We first observe that if $z \in \partial K$ and ${\mathcal C}_{\nu,z}^{\rho,\theta} \subset K$, then for all $\nu' \in \Sc^{N-1}$  
verifying $|\nu - \nu'| \leq \rho / (2 \theta)$, we have ${\mathcal C}_{\nu',z}^{\rho /  2,\theta} \subset K$. By compactness of  
$\Sc^{N-1}$, we can cover $\Sc^{N-1}$ with the traces on $\Sc^{N-1}$ of at most  
$p:={\b(N)}/{(\rho / (2 \theta))^{N-1}}$ balls of radius $\rho / (2 \theta)$ 
centered at $\nu_i$, for some positive constant $\b(N)$ and $1 \leq i \leq p$. Therefore,  
for any $z \in \partial K$, there exists $1\leq i \leq p$ such that ${\mathcal C}_{\nu_i,z}^{\rho / 2,\theta} \subset K$. \\ 
  
\noindent \textit{2. Local study of points of the boundary with the same interior cone axis.}  
We fix $1\leq i \leq p$ and set $A_i =\{ z \in \partial K; \; {\mathcal C}_{\nu_i,z}^{\rho / 2,\theta} \subset K \}$.  
Up to a rotation of $K$, we can assume that $\nu_i=(0,\dots,0,-1)=:\nu$.  
Let us fix $z \in A_i$, that we write $z=(x,y)$ with $x \in \R^{N-1}$ and $y \in \R$.  
Let us set $V=B_{N-1}(x,\rho / 4) \times \left( y-{\theta}/{2}, y+{\theta}/{2} \right)$ and  
$$  
D_i = \overline{V} \cap \underset{(x',y') \in A_i \cap \overline{V}}{\bigcup} {\mathcal C}_{\nu_i ,(x',y')}^{\rho / 2,\theta}.  
$$  
Then $A_i \cap V \subset \partial D_i \cap V$: indeed if $(x',y') \in A_i \cap V$, then $(x',y') \in D_i \cap V$,  
and $(x',y')$ can not lie in the interior of $D_i$, otherwise for $\l >0$ small enough,  
we would have $(x',y') -\l \nu \in D_i$, which would imply that $(x',y')$ lies in 
the interior of one of the cones forming $D_i$, and therefore in the interior of 
$K$, which is absurd since $(x',y') \in \partial K$. \\ 
  
\noindent \textit{3. The set $\partial D_i \cap V$ is a Lipschitz graph of  constant  
$\sqrt{(2\theta / \rho)^2-1}.$}  
More precisely let us prove that $\partial D_i \cap V$ is equal to  
\begin{eqnarray*}  
G_i & =& \big\{ (x',y') : x' \in B_{N-1}(x,\rho / 4) \\  
&& \hspace*{0.4cm}  
\text{and} \; y'=\max \{ y'': \; (x',y'')\in \partial {\mathcal C} \; \text{for one of 
  the cones} \;  {\mathcal C} \; \text{forming} \; D_i \} 
 \big\}.  
\end{eqnarray*}   
First of all, it is easy to show that $D_i$ is closed, and that the maximum in the definition of $G_i$ exists and is  
not equal to $y+\frac{\theta}{2}$; otherwise there would exist a cone ${\mathcal C}$ in 
$D_i$ such that $(x,y) \in {\rm int}({\mathcal C}) \subset {\rm int}(K)$, which is absurd. The  
inclusion $G_i \subset \partial D_i \cap V$ follows from the same argument used for the inclusion  
$A_i \cap V \subset \partial D_i \cap V$ in Step 2. Conversely, let us fix $(x',y') \in \partial D_i \cap V$.  
Then $(x',y') \in D_i$ since $D_i$ is closed, so that $(x',y')$ is included in the trace on $V$ of one of the  
cones forming $D$, let us say $(x',y') \in {\mathcal C}$. But then $(x',y')$ can not belong to ${\rm int}({\mathcal C})$, otherwise we would  
have $(x',y') \in {\rm int}(D_i)$, so we deduce that $(x',y') \in \partial {\mathcal C} \cap V$. Moreover  
if there exists $y'' >y'$  
such that $(x',y'')\in \partial {\mathcal C}'$  for some other of the cones ${\mathcal C}'$ forming $D_i$, then we must have $(x',y') \in 
{\rm int}({\mathcal C}') \cap V \subset {\rm int}(D_i)$, which is absurd, and proves that 
$y'$ is equal to the maximum in the definition of $G_i$, and that $\partial D_i \cap V \subset G_i$.  
Therefore $\partial D_i \cap V$ is a Lipschitz graph of constant $\mu= \sqrt{(2\theta / \rho)^2-1}$ as a supremum of  
graphs of cones of same parameters $\rho$ and $\theta$. \\ 
 
\noindent \textit{4. Estimate of the perimeter of $A_i$ in $V.$}  
It follows from Step 3 that $\partial D \cap V$ is $\Hc^{N-1}$ measurable with   
$$  
\Hc^{N-1}(\partial D \cap V) \leq \Lc^{N-1}(B_{N-1}(x,\rho / 4)) \, \sqrt{1+\mu^2},  
$$  
hence  
\begin{eqnarray*} 
\Hc^{N-1}(A_i \cap V) \leq \omega_{N-1} \left(\frac{\rho}{4}\right)^{N-1} \, \frac{2\theta}{\rho},  
\end{eqnarray*} 
where $\omega_j$ denotes the volume of the unit ball of $\R^j$. \\ 
  
\noindent \textit{5. Covering of $A_i$ with balls of fixed radius.} By Besicovitch's covering theorem 
(see \cite{eg92}), there exists a constant $\xi_N$ depending only on $N$ such that for any $\e >0$ and  
$R>0$, there exist numbers $\Gamma_1,\dots,\Gamma_{\xi_N}$ and a finite family $(x_{kj})$  
(for $1\leq k \leq \xi_N$ and $1 \leq j \leq \Gamma_k$) of points of $A_i \cap \bar{B}(0,R)$ such that   
$$  
\left\lbrace  
\begin{aligned}  
&A_i \cap \bar{B}(0,R)  \subset \bigcup_{k=1}^{\xi_N} \bigcup_{j=1}^{\Gamma_k} \bar{B}(x_{kj},\e),  \\  
& \text{for each}\; k, \text{the  balls} \; \bar{B}(x_{kj},\e), \;  1\leq j\leq \Gamma_k,\,
\text{are pairwise disjoint}.  
\end{aligned}  
\right.  
$$  
The family $(x_{kj})_j$ is a priori only countable, but has to be finite by boundedness of $A_i$ and because  
the radius of covering balls is fixed. We now want to estimate $\sum_{k=1}^{\xi_N} \Gamma_k$. Let us therefore compute  
$$  
\int_{K \cap \bar{B}(0,R+\e)} \sum_{k=1}^{\xi_N} \sum_{j=1}^{\Gamma_k} \1_{\bar{B}(x_{kj},\e)}.  
$$  
On the one hand, we have  
\begin{eqnarray} \label{abcd111} 
\sum_{k=1}^{\xi_N} \int_{K \cap \bar{B}(0,R+\e)} \sum_{j=1}^{\Gamma_k}\1_{\bar{B}(x_{kj},\e)}  
\leq \xi_N \Lc^N(K \cap \bar{B}(0,R+\e)),  
\end{eqnarray} 
because for each $k$, the balls $\bar{B}(x_{kj},\e)$ are pairwise disjoint. On the other hand, for each $k$ and $j$,  
the ball $\bar{B}(x_{kj},\e)$ contains a fixed portion of the cone ${\mathcal C}_{\nu_i ,x_{kj}}^{\rho/2,\theta}$, portion which is  
included in $K \cap \bar{B}(0,R+\e)$ by the interior cone property, since 
$x_{kj} \in A_i \cap \bar{B}(0,R)$. We call 
\begin{eqnarray*} 
\g:= \Lc^N(\bar{B}(x_{kj},\e)\cap {\mathcal C}_{\nu_i ,x_{kj}}^{\rho/2,\theta}) 
\end{eqnarray*} 
the volume of this portion of cone, the computation of which is done in 
Step 7. Note that $\gamma$ is independent of $x_{kj}.$ Therefore  
\begin{eqnarray} \label{abcd222} 
\int_{K \cap \bar{B}(0,R+\e)} \sum_{k=1}^{\xi_N} \sum_{j=1}^{\Gamma_k}\1_{\bar{B}(x_{kj},\e)}= \sum_{k=1}^{\xi_N}  
\sum_{j=1}^{\Gamma_k} \int_{K \cap \bar{B}(0,R+\e)} \1_{\bar{B}(x_{kj},\e)} 
 \geq \sum_{k=1}^{\xi_N} \Gamma_k \g. 
\end{eqnarray} 
From \eqref{abcd111} and  \eqref{abcd222}, we deduce 
$$  
\sum_{k=1}^{\xi_N} \Gamma_k \leq \frac{\xi_N}{\g} \Lc^N(K \cap \bar{B}(0,R+\e)).  
$$  
 Since $B((x,y),\e) \subset  
V=B_{N-1}(x,\rho / 4) \times \left( y-{\theta}/{2}, y+{\theta}/{2} \right),$ 
as soon as $\e < {\rm min}\{\rho / 4, {\theta}/{2}\}=\rho/4$, we deduce from this 
that $A_i \cap \bar{B}(0,R)$ can be covered by $\sum_{k=1}^{\xi_N} \Gamma_k$ 
cylinders of the form of $V$ centered at points of $A_i \cap \bar{B}(0,R)$, so that, 
from \eqref{abcd111}, 
$$  
\begin{aligned}  
\Hc^{N-1}(A_i \cap \bar{B}(0,R)) &\leq \sum_{k=1}^{\xi_N} \Gamma_k \, \omega_{N-1} 
\left(\frac{\rho}{4}\right)^{N-1} \, \frac{2\theta}{\rho} \\  
& \leq \frac{\xi_N}{\g}  \, \omega_{N-1} \left(\frac{\rho}{4}\right)^{N-1}  \,  
\frac{2\theta}{\rho} \, \Lc^{N}(K \cap B(0,R+\e)).  
\end{aligned}  
$$  
  
\noindent  \textit{6. Sum for all axes.} What we have done does not depend on the fixed 
direction axis $\nu_i$,  
and we know, thanks to Step 1 that $\partial K$ is the union of less than $p=\frac{\b(N)}{(\rho / 2 \theta)^{N-1}}$ sets of the form $A_i$, so that we finally have  
$$  
\Hc^{N-1}(\partial K \cap B(0,R)) \leq \frac{\b(N)}{(\rho / 2 \theta)^{N-1}} \, \frac{\xi_N}{\g}  \, {\omega_{N-1} \left(\frac{\rho}{4}\right)^{N-1}}  \, \frac{2\theta}{\rho} \, \Lc^{N}(K \cap B(0,R+\e))  
$$  
which gives \eqref{est-perim}. \\ 
 
\noindent  \textit{7. Computation of the value of $\gamma.$} 
As soon as $\e \leq \sqrt{\theta^2-(\rho/2)^2}$ (the length of the longest 
segment included in $ \partial {\mathcal C}_{\nu_i,x_{kj}}^{\rho/2,\theta}$), then $\bar{B}(x_{kj},\e)$ 
contains at least the straight portion of ${\mathcal C}_{\nu_i,x_{kj}}^{\rho/2,\theta}$ 
of length $l=\rho \mu \e /(2\theta)$, the volume of which equals  
$$  
\frac{\omega_{N-1}}{N} \frac{l^N}{\mu^{N-1}}=\frac{\omega_{N-1}}{N} \mu \left(\frac{\rho}{2\theta} \e \right)^N.  
$$  
This gives a lower bound for $\g$. Moreover, we obtain a more precise estimate 
for $\Lambda$ in \eqref{est-perim}: 
since $\rho < \theta$, we see that $\rho /4 \leq \sqrt{\theta^2-(\rho/2)^2}$, so that sending $\e$ to $\rho /4$, we get  
$$  
\Hc^{N-1}(\partial K \cap \bar{B}(0,R)) \leq 4^{N+1} N \b(N) \xi_N
\, \frac{1}{\rho} \frac{(\theta /\rho)^{2N}}{\sqrt{(2\theta / \rho)^2-1}} \; \Lc^{N}(K \cap B(0,\bar{B}+\rho /4)).
$$  
\end{proof}  
 
\subsection{Propagation of the interior cone property}  \label{ppg-cone}
  
We want to prove that the interior cone property is preserved for sets
whose evolution is governed by the Eikonal equation
\eqref{chjb}. We assume:\\

\noindent{\bf (H7)} The function $c(\cdot,t)$
is Lipschitz continuous with a constant independant of $t\in [0,T]$ and, for all $R>0,$ 
there exists an increasing modulus of continuity $\omega_R$ 
such that, for all $x\in B(0,R),$ $t,s\in [0,T],$ then
$$
|c(x,t)-c(x,s)|\leq \omega_R(|t-s|).
$$

\begin{theorem}\label{propagation-cone}  
Assume that $c$ satisfies {\bf (H1)} and {\bf (H7)} and that that $u_0$ 
satisfies {\bf (H6)}. Let $u$ be the unique uniformly continuous viscosity solution of \eqref{chjb}.
Then there exist $\rho >0$ and $\theta >0$ depending only on $K_0$, $N$, $\bar c$, $\underline c$ and $C$,  such 
that $K(t)=\{ x \in \R^N; \; u(x,t) \geq 0 \}$ has the interior cone property 
of parameters $\rho$ and $\theta$ for all $t \in [0,T]$. 
More precisely, let 
$r>0$ be such that $K_0$ has the interior ball property of radius $r>0$, then 
we can choose  
$$  
\theta=\min \left\{ 
\frac{\underline{c}^2}{6C\overline{c}},\underline{c}\, \omega_R^{-1} 
\left( \frac{\underline{c}}{4}
\right),r \right\} 
\quad \text{and} \quad 
\rho=\frac{\underline{c}}{2\overline{c}} \theta  
$$  
where $R>0$ is such that $K_0+\overline{c}T \bar{B}(0,1)\subset \bar{B}(0,R).$
\end{theorem}  

\begin{proof}[Proof of Theorem \ref{propagation-cone}.]  \ \\
\noindent{\it 1. Minimal time function.}
 We first remark that the assumption that $c(x,t)~\geq~\underline{c}$ implies 
that $t \mapsto u(x,t)$ is nondecreasing for any $x \in \R^N$. 
Moreover, this assumption and the finite speed of propagation 
property imply that if $u(x,t)=0$, then $u(x,s)>0$ for any $s\in (t,T]$.  
  Therefore, the minimal time function   
$$  
v(x) = \min \{ t \in [0,T]; \; u(x,t) \geq 0 \}  
$$  
is defined at points $x \in K(T)$, and for any $t \in [0,T]$,  
$$  
\begin{aligned}  
&\{ x \in \R^N; \; u(x,t) \geq 0 \} = \{ x \in \R^N; \; v(x) \leq t \},\\  
&\{ x \in \R^N; \; u(x,t) = 0 \} = \{ x \in \R^N; \; v(x) = t \}.  
\end{aligned}  
$$  
Moreover, $v$ is $1/ \underline{c}$-Lipschitz in $K(T)$: 
let us fix $x$ and $y$ in $K(T)$ with $v(x) \leq v(y)$. The function  
$$  
\overline{u}: (z,t) \mapsto 
\underset{|z'-z|\leq \underline{c} |t-v(x)|}{\sup}  u(z',v(x))   
$$  
is the unique uniformly continuous viscosity solution (see \cite{barles94}) of  
\begin{equation*}  
\left\lbrace  
\begin{array}{cl}  
\overline{u}_t(z,t)=\underline{c} |D\overline{u}(z,t)| 
& \text{in} \; \R^N \times (v(x),T),\\  
\overline{u}(\cdot,v(x))=u(\cdot,v(x)) 
& \text{in} \; \R^N. 
\end{array}  
\right.  
\end{equation*}  
The comparison principle for continuous viscosity solutions implies that $  
\overline{u} \leq u$ in $\R^N \times [v(x),T]$.  
In particular  
$$  
\overline{u}(y,\frac{1}{\underline{c}}|x-y| + v(x)) 
\leq u(y,\frac{1}{\underline{c}}|x-y| + v(x)),  
$$  
which implies by definition of $\overline{u}$ and $v$ that   
$$  
0 = u(x,v(x)) \leq \overline{u}(y,\frac{1}{\underline{c}}|x-y| + v(x)) \leq u(y,\frac{1}{\underline{c}}|x-y| + v(x)),  
$$  
from which the Lipschitz property follows, since we deduce that  
$$  
v(y) \leq \frac{1}{\underline{c}}|x-y| + v(x).  
$$  

\noindent{\it 2. Interior cone property at time $\bar{t}\in [\mu, T]$
  for some $\mu>0.$}
To prove the claim of the theorem, we will use arguments from control
theory. For this we need the velocity $c$ to be ${\mathcal C}^1$ in space, additionnal condition that we can assume without loss of generality 
by replacing $c$ by suitable space convolution $c_\delta$ of $c$. Then we get the result for $c_\delta$, and, letting $\delta\to0^+$, obtain the desired result since the constants $\theta$
and $\rho$ do not depend on $\delta$. \\

It is well-known that, for each time $t$, the set $K(t)$ can be seen as 
the reachable set from $K_0$ for the controlled system  
\begin{eqnarray} \label{sys-control}
x'(t)=c(x(t),t) {a}(t) \; \text{for} \; t \in [0,T],  
\end{eqnarray}
where the control ${a}$ takes its values in the unit closed ball. 
Let $x$ be an extremal trajectory, \textit{i.e.} a trajectory 
verifying $x(T) \in \partial K(T)$. For such a trajectory, it 
is easy to see that $t \mapsto u(x(t),t)$ is non-decreasing, 
from which we infer that $x(t) \in \partial K(t)$ for any 
$t \in [0,T]$, that is to say, $v(x(t))=t$. \\  
  
The Pontryagine maximum principle \cite{clarke76} implies the 
existence of an adjoint $p$ such that the following 
system is satisfied on $[0,T]$:  
\begin{equation}\label{pontryagine}  
\left\lbrace  
\begin{aligned}  
x'(t)&=c(x(t),t)\frac{p(t)}{|p(t)|},\\  
-p'(t)&=Dc(x(t),t)|p(t)|.  
\end{aligned}  
\right.  
\end{equation}
From now on, we fix $0\leq \bar{t}\leq T.$  
From \eqref{pontryagine} and the regularity of $c$ we infer that, 
if we set $M=3C\overline{c}$, then for any $s \in [0,\bar{t}]$,  
$$  
|x'(s) - x'(\bar{t})| \leq M (\bar{t}-s) + \omega_R (\bar{t}-s),  
$$  
where $R:=R_0+\oc T$ is given by Lemma \ref{fini-speed}.
By integration on $[t,\bar{t}],$ we deduce that, for any $t \in [0,\bar{t}]$,  
\begin{equation} \label{trajectory-C11}  
|x(\bar{t})-x(t)-x'(\bar{t})(\bar{t}-t)| \leq \frac M2 (\bar{t}-t)^2 
+ \omega_R(\bar{t}-t)(\bar{t}-t).  
\end{equation}  
Let $x \in \partial K(\bar{t})$, and 
let $x(\cdot)$ be an extremal trajectory with $x(\bar{t})=x$. 
We are going to show that for any $t \in [0,\bar{t}]$, the 
ball $\bar{B}(t)$ of radius $r(t)$ centered at $x(\bar{t})-x'(\bar{t})(\bar{t}-t)$ 
is contained in $K(\bar{t})$ for some $r(t)$ to determine, \textit{i.e.} 
that we have for any $\xi \in \bar{B}(0,r(t))$,  
$$  
v\left(x(\bar{t})-x'(\bar{t})(\bar{t}-t) + \xi \right) \leq \bar{t}.  
$$  
We therefore estimate, using the Lipschitz continuity of $v$ and \eqref{trajectory-C11},  
\begin{eqnarray*}  
&& \bar{t}-v\left( x(\bar{t})-x'(\bar{t})(\bar{t}-t) + \xi\right) \\  
&\geq & \bar{t}-v\left(x(\bar{t})
-x'(\bar{t})(\bar{t}-t)\right)- \frac{1}{\underline{c}} |\xi|\\  
&\geq & \bar{t}-v(x(t))- \frac{1}{\underline{c}} 
\left( \frac M2 (\bar{t}-t)^2 + \omega_R(\bar{t}-t)(\bar{t}-t) \right) 
- \frac{1}{\underline{c}} r(t) \\  
&= & \bar{t}-t- \frac{1}{\underline{c}} \left(\frac M2 (\bar{t}-t)^2 
+ \omega_R(\bar{t}-t)(\bar{t}-t) + r(t) \right).  
\end{eqnarray*}  
Thus if we set $r(t)=\frac{\underline{c}}{2}(\bar{t}-t)$,
the above quantity is nonnegative
as soon as  
$$
\bar{t}-t \leq \frac{\underline{c}}{2M} \quad \text{and} \quad \omega_R(\bar{t}-t) \leq
\frac{\underline{c}}{4}.
$$ 
For this choice, it follows
\begin{eqnarray*} 
\bar{B}(t)&= &\bar{B}\left(x(\bar{t})-x'(\bar{t})(\bar{t}-t),r(t)\right) \\
& = & \left\{ x(\bar{t})-\frac{x'(\bar{t})}{|x'(\bar{t})|}  |x'(\bar{t})|(\bar{t}-t) 
+\frac{\underline{c}}{2|x'(\bar{t})|} |x'(\bar{t})|(\bar{t}-t) \xi, \ \xi\in\bar{B}(0,1)
\right\} \\
&\subset&  K(\bar{t}).  
\end{eqnarray*}  
Since $x(\bar{t})=x$ and $\underline{c} \leq |x'(\bar{t})| \leq \overline{c}$, 
this proves the interior cone property at $x$ as soon as 
$\bar{t} \geq \mu = \min \left( \frac{\underline{c}}{2M}, \omega_R^{-1}\left(
\frac{\underline{c}}{4} \right) \right)$, 
of parameters   
$$  
\rho_1=\frac{\underline{c}}{2\overline{c}}\, \theta_1, \quad \text{with} 
\quad \theta_1=\min \left( \frac{\underline{c}^2}{2M},\underline{c}\omega_R^{-1} 
\left( \underline{c}/4 \right) \right).  
$$  
  
\noindent{\it 3. Interior cone property for small time $\bar{t}\in [0,\mu].$}
With the previous notation, let $x \in \partial K(\bar{t})$ and 
$x(\cdot)$ be an extremal trajectory of \eqref{sys-control} with $x(\bar{t})=x$. 
Let us recall that the regularity of $K_0$ implies that it 
has the interior ball property, \textit{i.e.} there exists $r>0$ 
independent of $y \in \partial K_0$ such that   
$$  
\bar{B}(y-\nu(y)r,r)\subset K_0,  
$$  
where $\nu(y)$ is the unit outer normal to $K_0$ at $y \in \partial K_0$. 
Note that, as a consequence, $K_0$ has the 
interior cone property  at $x(0)$ of parameters $\rho=r/2$ and $\theta=r$
and $\nu(x(0))=p(0)/|p(0)|$.
We see by the regularity of $K_0$ that $\nu(x(0))=p(0)/|p(0)|$, so that  
\begin{equation}\label{interior-ball}  
\bar{B}(x(0)-\frac{p(0)}{|p(0)|}r,r) \subset K_0.  
\end{equation}  
We will prove that, for $\bar{t} \leq \mu$, $K(\bar{t})$ has the interior cone 
property of parameters $\rho=r/2$ and $\theta=r$. 
Let $y \in {\mathcal C}_{\nu,x}^{r/2,r}$ 
with $\nu=-\frac{p(\bar{t})}{|p(\bar{t})|}$. We write $y$ as  
\begin{equation}\label{interior-cone1}  
y=x-\frac{p(\bar{t})}{|p(\bar{t})|}\lambda    +\frac 12 \lambda \xi,  
\end{equation}  
where $0\leq \lambda \leq r$ and $|\xi|\leq 1$. Let $y(\cdot)$ be the solution of  
\begin{equation*}  
\left\lbrace  
\begin{aligned}  
y'(t)&=c(y(t),t)\frac{p(t)}{|p(t)|} \quad \text{for} \; t \in [0,\bar{t}],\\  
y(\bar{t})&=y,  
\end{aligned}  
\right.  
\end{equation*}
where $p(\cdot)$ is the adjoint associated with $x(\cdot)$ by
\eqref{pontryagine}.
It is enough to prove that $y(0) \in K_0$, since then $y=y(\bar{t}) \in
K(\bar{t})$. 
Because of \eqref{interior-ball}, we only have to show that  
$$  
\left| y(0)-\left(x(0)-\frac{p(0)}{|p(0)|}\lambda \right) \right| \leq \lambda.  
$$  
Moreover, we remark that \eqref{interior-cone1} implies that  
$$  
\left| y(\bar{t})-\left(x(\bar{t})
-\frac{p(\bar{t})}{|p(\bar{t})|}\lambda\right)\right| 
= |\frac 12 \lambda \xi | \leq \frac{\lambda}{2}.  
$$  
Let us therefore set   
$$  
f(t)=|y(t)-x(t)+ \lambda \frac{p(t)}{|p(t)|}|^2,  
$$  
so that $f(\bar{t}) \leq \frac{\lambda^2}{4}$. 
It only remains to prove that $f(0) \leq \lambda^2$. But   
\begin{eqnarray*}   
f'(t)&=&
2\left\langle y(t)-x(t), y'(t)-x'(t) \right\rangle 
+ 2 \lambda \left\langle y'(t)-x'(t), \frac{p(t)}{|p(t)|} \right\rangle \\
&& \quad\quad\quad\quad\quad 
+  2 \lambda \left\langle y(t)-x(t), 
\frac{d}{dt} \frac{p(t)}{|p(t)|} \right\rangle \\  
&=& 2\left\langle y(t)-x(t), \left(c(y(t),t)-c(x(t),t)\right)
\frac{p(t)}{|p(t)|} \right\rangle \\
&& \quad\quad\quad\quad\quad 
+  2 \lambda \left\langle \left(c(y(t),t)-c(x(t),t)\right)\frac{p(t)}{|p(t)|},
\frac{p(t)}{|p(t)|} \right\rangle \\ 
& & \quad\quad\quad \quad\quad 
+ 2 \lambda \left\langle y(t)-x(t), \frac{p'(t)}{|p(t)|}
-\frac{p(t)\left\langle p(t),p'(t) \right\rangle }{|p(t)|^3} \right\rangle\\  
&\geq&  -2C|y(t)-x(t)|^2 -2\lambda C |y(t)-x(t)| -2\lambda |y(t)-x(t)| 
\left|\frac{p'(t)}{|p(t)|}\right| \\ 
& &  \quad\quad\quad\quad\quad 
-2\lambda |y(t)-x(t)|\left| 
\frac{p(t)\left\langle p(t),p'(t) \right\rangle }{|p(t)|^3}\right|.  
\end{eqnarray*}  
Thanks to \eqref{pontryagine}, we know that  
$$  
\left|\frac{p'(t)}{|p(t)|}\right|\leq C 
\quad \text{and} \quad 
\left|\frac{p(t)\langle p(t),p'(t) \rangle }{|p(t)|^3}\right|\leq C,  
$$  
so that  
$$  
f'(t) \geq -2C|y(t)-x(t)|^2 -6 \lambda C |y(t)-x(t)|.  
$$  
But if we set $g(t)=|y(t)-x(t)|^2$, then   
$$  
g'(t)=2\langle y(t)-x(t), y'(t)-x'(t) \rangle \geq -2C|y(t)-x(t)|^2=-2Cg(t),  
$$  
which implies that for all $t \in [0,\bar{t}]$  
$$  
g(t)e^{2Ct} \leq g(\bar{t})e^{2C\bar{t}},  
$$  
that is to say thanks to \eqref{interior-cone1}  
$$  
|y(t)-x(t)| \leq |y-x|e^{C(\bar{t}-t)} \leq \frac{3\lambda}{2}e^{C\bar{t}}.  
$$  
We therefore obtain  
$$  
f'(t) \geq -2C \left( \frac{3\lambda}{2}e^{C\bar{t}} \right)^2 -6 \lambda C
\, 
\frac{3\lambda}{2}e^{C\bar{t}}
=-\left(\frac 92 C e^{2C\bar{t}} + 9Ce^{C\bar{t}} \right) \lambda^2.
$$   
If we set $k=\frac 92 C e^{2C\bar{t}} + 9Ce^{C\bar{t}}$, we finally have  
$$  
f(0)\leq f(\bar{t}) +k \lambda^2 \bar{t} \leq \frac{\lambda^2}{4} 
+k \lambda^2 \bar{t}\leq \lambda^2  
$$  
as soon as $k\bar{t} \leq \frac 34$. Thus if we set ${b}$ to be the unique
solution of $\frac 92 {b} e^{2{b}} + 9{b}e^{{b}}= \frac 34$ (${b}>0$), 
we get that $f(0)\leq 0$ as soon as $\bar{t} \leq {b}/C$. If we assume that   
$$  
\frac {b}C \geq \frac{\underline{c}}{2M}=\frac{\underline{c}}{6C\overline{c}},  
$$  
which is always possible by reducing $\underline{c}$ 
or increasing $\overline{c}$, we see
that $K(\bar{t})$ 
has the interior cone property of parameters $\rho_2=r/2$ and
$\theta_2=r$ 
for all $0\leq \bar{t}\leq \mu$ 
(note that the parameters $\rho_2, \theta_2$ depend
only on $c$ and $K_0$).\\

\noindent{\it 4. End of the proof.} 
We remark that   
$$  
\frac{\rho_1}{\theta_1} = \frac{\underline{c}}{2\overline{c}} 
\leq \frac 12 = \frac{\rho_2}{\theta_2},  
$$  
whence we finally obtain that for any $\bar{t} \geq 0$, $K(\bar{t})$ 
has the interior cone property of parameters 
$\rho=\frac{\underline{c}}{2\overline{c}} \theta$ 
with $\theta=\min \{\theta_1,\theta_2\}$.  
\end{proof}

\end{document}